\documentclass[11pt]{article}
\usepackage[margin=1in]{geometry} 
\usepackage{amssymb,amsfonts,amsmath,bbm,mathrsfs,stmaryrd}
\usepackage{xcolor}
\usepackage{url}
 \usepackage{relsize}

\usepackage[colorlinks,
             linkcolor=black!75!red,
             citecolor=blue,
             pdftitle={BU},
             pdfauthor={Alexandru Chirvasitu, Benjamin Passer},
             pdfproducer={pdfLaTeX},
             pdfpagemode=None,
             bookmarksopen=true
             bookmarksnumbered=true]{hyperref}

\usepackage{tikz}
\usetikzlibrary{arrows,calc,decorations.pathreplacing,decorations.markings,intersections,shapes.geometric,through,fit,shapes.symbols,positioning,decorations.pathmorphing}

\usepackage[amsmath,thmmarks,hyperref]{ntheorem}
\usepackage{cleveref}

\crefname{section}{Section}{Sections}
\crefformat{section}{#2Section~#1#3} 
\Crefformat{section}{#2Section~#1#3} 

\crefname{subsection}{\S}{\S\S}
\crefformat{subsection}{#2\S#1#3} 
\Crefformat{subsection}{#2\S#1#3} 

\theoremstyle{plain}

\newtheorem{lemma}{Lemma}[section]
\newtheorem{proposition}[lemma]{Proposition}
\newtheorem{corollary}[lemma]{Corollary}
\newtheorem{theorem}[lemma]{Theorem}
\newtheorem{conjecture}[lemma]{Conjecture}
\newtheorem{question}[lemma]{Question}

\theoremstyle{nonumberplain}

\theoremstyle{plain}
\theorembodyfont{\upshape}
\theoremsymbol{\ensuremath{\blacklozenge}}

\newtheorem{claim}[lemma]{Claim}

\crefname{definition}{definition}{definitions}
\crefformat{definition}{#2definition~#1#3} 
\Crefformat{definition}{#2Definition~#1#3} 

\crefname{ex}{example}{examples}
\crefformat{example}{#2example~#1#3} 
\Crefformat{example}{#2Example~#1#3} 

\crefname{remark}{remark}{remarks}
\crefformat{remark}{#2remark~#1#3} 
\Crefformat{remark}{#2Remark~#1#3} 

\crefname{convention}{convention}{conventions}
\crefformat{convention}{#2convention~#1#3} 
\Crefformat{convention}{#2Convention~#1#3} 

\crefname{claim}{claim}{claims}
\crefformat{claim}{#2claim~#1#3} 
\Crefformat{claim}{#2Claim~#1#3} 

\crefname{conjecture}{conjecture}{conjectures}
\crefformat{conjecture}{#2conjecture~#1#3} 
\Crefformat{conjecture}{#2Conjecture~#1#3}

\crefname{lemma}{lemma}{lemmas}
\crefformat{lemma}{#2lemma~#1#3} 
\Crefformat{lemma}{#2Lemma~#1#3} 

\crefname{proposition}{proposition}{propositions}
\crefformat{proposition}{#2proposition~#1#3} 
\Crefformat{proposition}{#2Proposition~#1#3} 

\crefname{question}{question}{questions}
\crefformat{question}{#2question~#1#3} 
\Crefformat{question}{#2Question~#1#3}

\crefname{corollary}{corollary}{corollaries}
\crefformat{corollary}{#2corollary~#1#3} 
\Crefformat{corollary}{#2Corollary~#1#3} 

\crefname{theorem}{theorem}{theorems}
\crefformat{theorem}{#2theorem~#1#3} 
\Crefformat{theorem}{#2Theorem~#1#3} 

\crefname{assumption}{assumption}{Assumptions}
\crefformat{assumption}{#2assumption~#1#3} 
\Crefformat{assumption}{#2Assumption~#1#3} 

\crefname{equation}{}{}
\crefformat{equation}{(#2#1#3)} 
\Crefformat{equation}{(#2#1#3)}

\theoremstyle{nonumberplain}
\theoremsymbol{\ensuremath{\blacksquare}}

\newtheorem{proof}{Proof}

\newcommand\bC{{\mathbb C}}

\newcommand\bQ{{\mathbb Q}}
\newcommand\bR{{\mathbb R}}
\newcommand\bS{{\mathbb S}}
\newcommand\bT{{\mathbb T}}
\newcommand\bZ{{\mathbb Z}}

\newcommand\cC{{\mathcal C}}

\newcommand\cO{{\mathcal O}}
\newcommand\cP{{\mathcal P}}

\def\polhk#1{\setbox0=\hbox{#1}{\ooalign{\hidewidth
    \lower1.5ex\hbox{`}\hidewidth\crcr\unhbox0}}}

\title{Compact Group Actions on Topological and Noncommutative Joins}
\author{Alexandru Chirvasitu\footnote{University at Buffalo, \url{achirvas@buffalo.edu} \hspace{2 pt} Partial support from NSF grant DMS 1565226 and the 2016 Simons Semester in Noncommutative Geometry through the Simons-Foundation grant 346300 and the Polish Government MNiSW 2015-2019 matching fund.}, Benjamin Passer\footnote{Technion-Israel Institute of Technology, \url{benjaminpas@technion.ac.il} \hspace{2 pt} Partial support from NSF grants DMS 1300280 and DMS 1363250, a Zuckerman Fellowship at the Technion, EU grant H2020-MSCA-RISE-2015-691246-QUANTUM DYNAMICS, and the 2016 Simons Semester in Noncommutative Geometry through the Simons-Foundation grant 346300 and the Polish Government MNiSW 2015-2019 matching fund.}}

\begin{document}

\date{}

\maketitle

\begin{abstract}
  We consider the Type 1 and Type 2 noncommutative Borsuk-Ulam
  conjectures of Baum, D\polhk{a}browski, and Hajac: there are no
  equivariant morphisms $A \to A \circledast_\delta H$ or
  $H \to A \circledast_\delta H$, respectively, when $H$ is a
  nontrivial compact quantum group acting freely on a unital
  $C^*$-algebra $A$. Here $A \circledast_\delta H$ denotes the
  equivariant noncommutative join of $A$ and $H$; this join procedure
  is a modification of the topological join that allows a free action
  of $H$ on $A$ to produce a free action of $H$ on
  $A \circledast_\delta H$. For the classical case $H = \cC(G)$, $G$ a
  compact group, we present a reduction of the Type 1
  conjecture and counterexamples to the Type 2 conjecture. We also present 
some examples and conditions under which the Type 2 conjecture does hold.
\end{abstract}


\vspace{.5cm}


\section{Introduction}\label{se.intro}

The join of two topological spaces $X$ and $Y$ is a quotient
$X * Y = X \times Y \times [0,1] / \sim$, where the equivalence
relation makes identifications at the endpoints of $[0, 1]$. Namely,
if $x_0 \in X$ is fixed, then all points $(x_0, y, 1)$ are identified,
and if $y_0 \in Y$ is fixed, all points $(x, y_0, 0)$ are
identified. In \cite{BDH}, the authors conjecture that the topological
join, and a $C^*$-algebraic variant thereof, may be used to greatly
generalize the Borsuk-Ulam theorem. Their topological conjecture is as
follows.

\begin{conjecture}\label{conj.BDH}
  Suppose $G$ is a nontrivial compact group acting freely and
  continuously on a compact Hausdorff space $X$. Then there is no
  equivariant, continuous map from $X*G$ to $X$, where $X*G$ is
  equipped with the diagonal action
  $[(x, g, t)] \cdot h = [(x\cdot h, gh, t)]$.
\end{conjecture}

The conjecture generalizes the Borsuk-Ulam theorem in that if
$G = \bZ/2$, $X = \mathbb{S}^{k}$, and the group action is given by
the antipodal map $x \mapsto -x$, then the conjecture reads as the
Borsuk-Ulam theorem itself: there is no odd function from $\bS^{k+1}$
to $\bS^{k}$. Of particular interest is the fact that the sphere
$\bS^{k}$ is an iterated join of $k + 1$ copies of $\bZ/2$, and the
antipodal action is compatible with this join process. As remarked in
\cite{BDH}, \Cref{conj.BDH} holds when $X = G * G$ due to
non-contractibility of $G$. Further, \cite[Corollary
3.1]{volovikovindex} shows that \Cref{conj.BDH} holds when $G$ has a
nontrivial torsion element (see also \cite[Proposition 4.1]{volovikov}
for a description of this proof).

If $A$ and $B$ are unital $C^*$-algebras, the \lq\lq
classical\rq\rq\hspace{0pt} \textit{noncommutative join}
\begin{equation}
  A \circledast B = \{f \in \cC([0,1], A\otimes B) \hspace{5pt} |\ f(0)\in \bC \otimes B,\ f(1)\in \bC \otimes A\}
\end{equation}
of \cite{joinfusion} and \cite{freeness} directly generalizes the join $X*Y$ of compact Hausdorff spaces. Here we have used $\otimes$ to refer to the minimal tensor product, as we do in the rest of the manuscript. However, if $(H, \Delta)$ is a compact quantum group with a free coaction $\delta: A \to A \otimes H$, questions about this coaction are more suited to the \textit{equivariant join}
\begin{equation}\label{eq:nc}
  A \circledast_\delta H =\{f \in \cC([0,1], A\otimes H) \hspace{5pt} |\ f(0)\in \bC \otimes H,\ f(1)\in \delta(A)\},
\end{equation}
which admits a free coaction $\delta_\Delta$ generated by
$\mathrm{id} \otimes \Delta: \cC([0, 1], A) \otimes H \to \cC([0, 1],
A) \otimes H \otimes H$ (see \cite[Theorem 1.5]{BDH}). Conjecture 2.3
of \cite{BDH}, which we repeat here, generalizes \Cref{conj.BDH} to
the quantum setting.

\begin{conjecture}\label{conj.BDH2}
  Suppose $A$ is a unital $C^*$-algebra with a free coaction of a
  nontrivial compact quantum group $(H, \Delta)$.

  Type 1. There does not exist a $(\delta, \delta_\Delta)$-equivariant
  $*$-homomorphism $A \to A \circledast_\delta H$.

  Type 2. There does not exist a $(\Delta, \delta_\Delta)$-equivariant
  $*$-homomorphism $H \to A \circledast_\delta H$.
\end{conjecture}

Note that a coaction $\delta: A \to A \otimes H$ of $(H, \Delta)$ is
\textit{free} when it satisfies the following condition from
\cite{ellwood}, which is written succinctly as \cite[(1.10)]{BDH}:
\begin{equation}\label{eq:cofree}
  \overline{\left\{ \sum\limits_\mathrm{finite} (a_i \otimes 1)\delta(b_i): a_i, b_i \in A \right\} } = A \otimes H.
\end{equation}
Further, if $A$ and $B$ are $C^*$-algebras with coactions $\delta_A$,
$\delta_B$ of $(H, \Delta)$, then a morphism $\phi: A \to B$ is
\textit{equivariant} if
\begin{equation}\label{eq.equivariance}
  \begin{tikzpicture}[auto,baseline=(current  bounding  box.center)]
    \path[anchor=base] (0,0) node (a) {$A$} +(3,.5) node (b) {$B$} +(6,0) node (bh) {$B \otimes H$} +(3,-.5) node (ah) {$A \otimes H$};
         \draw[->] (a) to[bend left=10] node[pos=.5,auto] {$\scriptstyle \phi$}  (b);
         \draw[->] (a) to[bend right=10] node[pos=.5,auto,swap] {$\scriptstyle \delta_A$} (ah);
         \draw[->] (b) to[bend left=10] node[pos=.5,auto] {$\scriptstyle \delta_B$} (bh);
         \draw[->] (ah) to[bend right=10] node[pos=.5,auto,swap] {$\scriptstyle \phi \otimes \mathrm{id}$} (bh);
  \end{tikzpicture}
\end{equation}
commutes.

As a consequence of \cite[Corollary 2.4]{saturated}, \Cref{conj.BDH2}
Type 1 holds when $H = \cC(G)$ for $G$ a compact group with a
nontrivial torsion element. Specifically, \cite[Corollary
2.4]{saturated} is a reduction to the topological case in
\cite{volovikovindex}, and some alternative (but still very much
related) proof strategies are described in the remainder of
\cite{saturated}. No counterexamples to Type 1 in its full generality
are known, but in \Cref{se.tp2}, we show that Type 2 counterexamples
exist for the \lq\lq quantum\rq\rq \hspace{0pt} group
$\cC(\mathbb{S}^1)$ acting on certain noncommutative $C^*$-algebras,
which can even be separable and nuclear. On the other hand, there are
also some restrictive conditions, applicable in both the classical
group and quantum group settings, under which the Type 2 conjecture
holds. In \Cref{se.tp1}, we deal exclusively with the classical case
$H = \cC(G)$ and present a reduction of the Type 1 conjecture in this
setting. We also consider related questions regarding eigenfunctions
in $\cC(\mathbb{Z}_p^{*n})$, including the limit as $n \to \infty$,
based upon questions in \cite{saturated} that attempt to use
eigenfunctions to generalize a strategy in \cite{volovikovindex}.


\section{Type 2: Counterexamples and Special Cases}\label{se.tp2}

In this section, we address \Cref{conj.BDH2} Type 2. For a coaction
$\delta:A\to A\otimes H$ and a simple (hence finite-dimensional)
$H$-comodule $\rho:V\to V\otimes H$, we let $A_\rho$ denote the {\it
  $\rho$-isotypic subspace} of $A$:
\begin{equation*}
  A_\rho = \text{ sum of the images of all }H\text{-}\text{equivariant maps }V\to A. 
\end{equation*}
The sum of all $A_\rho$ (as $\rho$ ranges over all simple
$H$-comodules) is a dense $*$-subalgebra of $A$, referred to in
\cite[$\S$1.2]{BDH} as the {\it Peter-Weyl subalgebra} $\cP_H(A)$ of
$A$. It can be recast, as in \cite{BDH}, as the preimage through
$\rho$ of $A\otimes_{\mathrm{alg}}\cO(H)$, where the latter symbol
denotes the unique dense Hopf $*$-subalgebra of $H$. Note that we have
\begin{equation*}
  \cO(H)=\bigoplus_\rho H_\rho,
\end{equation*}
where $H$ coacts on itself on the right via the comultiplication
$\Delta:H\to H\otimes H$. Equivalent formulations of freeness were
studied by the authors of \cite{freeness}; below we present a special
case of their results, with a proof for completeness.

 \begin{proposition}\label{prop.saturation}(special case of \cite[Theorem 0.4]{freeness})
   A coaction $\delta:A\to A\otimes H$ is free if and only if it is
   saturated, i.e. for every simple $H$-comodule $\rho$, the unit
   $1 \in A$ belongs to $\overline{A_\rho^* A_\rho}$.
\end{proposition}
\begin{proof} {\bf ($\Rightarrow$)} As in \cite[Definition
  0.1]{freeness}, freeness means that the linear span of elements of
  the form
  \begin{equation*}
    (x\otimes 1)\delta(y),\ x,y\in A
  \end{equation*}
  is dense in $A\otimes H$. It suffices to let $x,y$ range over the
  dense subalgebra $\cP_H(A) = \bigoplus_\rho A_\rho$ instead, and
  since we have
  $\delta(\cP_H(A))\subset \bigoplus_\rho A_\rho\otimes H_\rho$, it follows that the linear span
\begin{equation}\label{eq:eta_ro}
  \sum_{\eta,\rho} A_\eta^* A_\rho\otimes H_\rho
\end{equation}
is dense in $A\otimes H$.

Since the subspaces $H_\rho$ and $H_{\rho'}$ of $H$ are orthogonal
with respect to the Haar state $h:H\to \bC$ when $\rho\ne \rho'$, any
element in $A\otimes H_\rho$ can be approximated arbitrarily well by
elements in
\begin{equation*}
  \sum_\eta A_\eta^* A_\rho\otimes H_\rho
\end{equation*}
alone (i.e., in \Cref{eq:eta_ro} we fix $\rho$ and allow only $\eta$
to vary). In particular, $1\in A$ is in the closure of
$\sum_\eta A_\eta^* A_\rho$. The conclusion now follows from the fact
that for $\eta\ne\rho$, the product $A_\eta^* A_\rho$ is contained in
the sum of all $A_\mu$, $\mu\ne 1$ (the trivial $H$-comodule), whereas
$1$ is in the $H$-fixed subspace
$A_1 = \{a\in A\ |\ \delta(a) = a\otimes 1\}$.

{\bf ($\Leftarrow$)} Suppose that for every simple
$H$-comodule $\rho$, the unit $1 \in A$ belongs to
$\overline{A_\rho^* A_\rho}$. This implies that for every $\rho$,
$1\otimes H_\rho$ is in the closure of $A_\rho^*A_\rho\otimes H_\rho$,
and hence $A\otimes H_\rho$ is in the closure of
$AA_\rho\otimes H_\rho$. The conclusion that the action is free then follows from the observation that for
every simple $H$-comodule $\rho$, $A_\rho\otimes H_\rho$ is contained
in $\delta(A_\rho) \subset \delta(A)$.
\end{proof}

When $(H, \Delta)$ is the Hopf $C^*$-algebra $\cC(G)$ of a compact
group $G$ with comultiplication dual to the multiplication of $G$, see
also \cite[Definition 5.2, Theorem 5.10]{phillips}. We will henceforth
focus almost exclusively on the case $H = \cC(G)$, where equivariance
in the sense of \Cref{eq.equivariance} is equivalent to the usual
$G$-equivariance of morphisms. Further, in this setting it is known
from \cite{BDH} that there is no distinction between the classical
join $A \circledast \cC(G)$ and the equivariant join
$A \circledast_\delta \cC(G)$. Indeed, applying the map
\begin{equation}\label{eq.2joins}
  \begin{tikzpicture}[auto,baseline=(current  bounding  box.center)]
    \path[anchor=base] (0,0) node (1) {$A\otimes \cC(G)$} +(4,0) node (2) {$A\otimes \cC(G)\otimes \cC(G)$} +(8,0) node (3) {$A\otimes \cC(G)$};
         \draw[->] (1) to node[pos=.5,auto] {$\scriptstyle \delta\otimes\mathrm{id}_{\cC(G)}$}  (2);
         \draw[->] (2) to node[pos=.5,auto] {$\scriptstyle \mathrm{id}_A\otimes\mathrm{mult}$}  (3);
  \end{tikzpicture}
\end{equation}
pointwise on $\cC([0,1], A \otimes \cC(G))$ identifies the subspace $A \circledast \cC(G)$ with the subspace $A\circledast_\delta \cC(G)$, intertwining the diagonal $G$-action on the left hand side and the right-tensorand regular action on the right hand side.

When a compact group $G$ acts freely on a compact Hausdorff space $X$, any orbit of $X$ gives an equivariant embedding $G \hookrightarrow X$. However, there is no analogous phenomenon in the $C^*$-algebraic setting: a simple $C^*$-algebra $A$ may have a free $\cC(G)$-coaction even though $A$ can have no quotient isomorphic to $\cC(G)$. The embedding $G \hookrightarrow X$ is exactly why  \Cref{conj.BDH} was not split into two types, and we will exploit this difference to produce counterexamples to \Cref{conj.BDH2} Type 2. First, note that when $H = \cC(G)$ is classical, the conjecture may be rephrased in a way that avoids the join entirely.

\begin{lemma}\label{le.tech}
  Let $G$ be a compact group acting on a unital $C^*$-algebra $A$. There is an equivariant unital $*$-homomorphism $\cC(G) \to A \circledast \cC(G) \cong A \circledast_\delta \cC(G)$ if and only if both conditions hold:
  \begin{itemize}
    \item there is a $G$-equivariant unital $*$-homomorphism $\varphi:\cC(G)\to A$, and 
    \item $\varphi_1=\varphi$ can be connected to a one-dimensional representation $\varphi_0:\cC(G)\to \bC\subset A$ through a path $\varphi_t$, $t\in [0,1]$ in $\mathrm{Hom}(\cC(G),A)$.
  \end{itemize}
Further, the first condition guarantees that the associated coaction $\delta: A \to A \otimes \cC(G)$ is free.
\end{lemma}
\begin{proof}
If there is an equivariant map $\psi: \cC(G) \to A \circledast \cC(G)$, then evaluation at any $t \in [0, 1]$ produces equivariant maps $\psi_t: \cC(G) \to A \otimes \cC(G) \cong \cC(G, A)$. From the boundary conditions of the (classical) join, $\psi_1$ maps to constant $A$-valued functions, and $\psi_0$ maps to $\cC(G)$. The equivariant map $\varphi = \mathrm{ev}_{e} \circ \psi_1$ is then connected within $\mathrm{Hom}(\cC(G), A)$ via $\mathrm{ev}_{e} \circ \psi_{t}$ to a one-dimensional representation, $\mathrm{ev}_e \circ \psi_0$.

If instead we assume that the conditions hold, we consider the two conclusions separately.

\vspace{.07 in}

{\bf 1: Freeness.} 
The comultiplication $\Delta$ on $\cC(G)$ defines a free coaction of $\cC(G)$ on itself, as the closure of finite sums in \Cref{eq:cofree} is actually a closed $*$-subalgebra of $\cC(G) \otimes \cC(G) \cong \cC(G \times G)$ on which the Stone-Weierstrass theorem applies (see \cite{freeness}). Moreover, we have a unital $*$-homomorphism $\phi: \cC(G) \to A$ that is $G$-equivariant, so it satisfies the coaction equivariance identity $(\phi \otimes \mathrm{id}) \circ \Delta = \delta \circ \phi$. Fix $\varepsilon > 0$, $a \in A \setminus \{0\}$, $f \in \cC(G)$, and finitely many elements $h_i, k_i \in \cC(G)$, $1 \leq i \leq m$, so that
\begin{equation*} 
\left|\left| \sum\limits_{i=1}^m (h_i \otimes 1)\Delta(k_i) - 1 \otimes f \right|\right| < \frac{\varepsilon}{||a||}
\end{equation*}
by freeness of the comultiplication $\Delta$. Applying $\phi \otimes \mathrm{id}$ and then left multiplying by $a \otimes 1$ yields
\begin{equation*}
\left|\left| \sum\limits_{i=1}^m (\phi(h_i) \otimes 1)\delta(\phi(k_i)) - 1 \otimes f \right|\right| < \frac{\varepsilon}{||a||} \implies \left|\left| \sum\limits_{i=1}^m (a \phi(h_i) \otimes 1)\delta(\phi(k_i)) - a \otimes f \right|\right| < \varepsilon,
\end{equation*}
so the closure in \Cref{eq:cofree} includes any $a \otimes f$. The closed span of such elements is then all of $A \otimes \cC(G)$.

\vspace{.07 in}

{\bf 2: Existence of map into the join.} Recasting the join in an equivalent form
\begin{equation}\label{eq:join_bis}
  A \circledast \cC(G) \cong \{f \in \cC([0,1] \times G, A)\ |\ f|_{\{0\}\times G}\text{ takes values in }\bC\subseteq A,\ f|_{\{1\}\times G}\text{ is constant on }G\},
\end{equation}
we can define $\psi:\cC(G)\to A \circledast \cC(G)$ so that it sends $f\in \cC(G)$ to the function
\begin{equation*}
  \psi(f):[0,1]\times G\to A,\ \psi(f)(t,g) =  g^{-1}\triangleright \varphi_t(f(\bullet g)),
\end{equation*}
where $f(\bullet g)$ is the function $x\mapsto f(xg)$ and $\triangleright$ denotes the group action. It is now immediate to check that 
\begin{itemize}
  \item $\psi(f)|_{\{0\}\times G}$ is in $\cC(G)\subset A\otimes \cC(G)$, as $\varphi_0$ was assumed to be a one-dimensional representation;
  \item $\psi(f)|_{\{1\}\times G}$ is constant on $G$ and hence represents a single element of $A\subset A\otimes \cC(G)$, as $\varphi=\varphi_1:\cC(G)\to A$ was assumed to be equivariant;
  \item $\psi$ is $G$-equivariant with respect to the $G$-action on $\cC([0,1]\times G,A)$ defined by
    \begin{equation*}
      (h\triangleright \zeta)(t,g) = h\triangleright \zeta(t,gh). 
    \end{equation*}
\end{itemize}
In other words, $\psi$ is the desired equivariant map $\cC(G)\to A \circledast \cC(G)$. 
\end{proof}

Equipped with \Cref{le.tech}, we can find counterexamples to the Type 2 conjecture.

\begin{theorem}\label{th.type2failure}
Let $\mathcal{H}$ be an infinite-dimensional, separable Hilbert space and set $A = B(\mathcal{H}) \otimes \cC(\mathbb{S}^1)$. If $\Delta: \cC(\mathbb{S}^1) \to \cC(\mathbb{S}^1) \otimes \cC(\mathbb{S}^1)$ denotes the comultiplication, then define a coaction $\delta: A \to A \otimes \cC(\mathbb{S}^1)$ by $\delta = \mathrm{id} \otimes \Delta$. This coaction is free, and there is an equivariant, unital $*$-homomorphism $\psi: \cC(\mathbb{S}^1) \to A \circledast_\delta  \cC(\mathbb{S}^1)$, so \Cref{conj.BDH2} Type 2 fails.
\end{theorem}
\begin{proof}
  First, there is an equivariant map
  $\varphi = \varphi_1: \cC(\mathbb{S}^1) \to A$ defined by
  $\phi(f) = 1 \otimes f$. Because $\cC(\mathbb{S}^1)$ is the
  universal $C^*$-algebra generated by a single unitary character
  $\chi: z \mapsto z$, the map $\varphi_1$ is determined by
  $\varphi_1(\chi) = 1 \otimes \chi$. The unitary group
  $U(\mathcal{H})$ with the norm topology is simply connected by
  Kuiper's theorem (\cite{kuiper}), so the unitary group
  $U(A) \cong U(\cC(\mathbb{S}^1, B(\mathcal{H}))) \cong
  \cC(\mathbb{S}^1, U(\mathcal{H}))$ is path connected. Let $u_t$,
  $t \in [0, 1]$ be a continuous path of unitaries in $A$ connecting
  $u_1 = I \otimes \chi$ to $u_0 = I \otimes 1$, and define a
  continuous path of unital $*$-homomorphisms
  $\varphi_t: \cC(\mathbb{S}^1) \to A$ by $\varphi_t(\chi) =
  u_t$. Because $\varphi_0$ is just a one-dimensional representation,
  we may apply \Cref{le.tech}, guaranteeing freeness of the coaction
  on $A$ and existence of an equivariant map
  $\cC(\mathbb{S}^1) \to A \circledast_\delta \cC(\mathbb{S}^1)$.
\end{proof}

To see the result of the above proof more explicitly, note that the final equivariant map $\psi: \cC(\mathbb{S}^1) \to A \circledast_\delta  \cC(\mathbb{S}^1) = (B(\mathcal{H}) \otimes \cC(\mathbb{S}^1)) \circledast_\delta  \cC(\mathbb{S}^1)$ is determined by 
\begin{equation*}
\psi(\chi)[t] = u_t \otimes \chi, \hspace{.3 cm} t \in [0,1],
\end{equation*}
so its general form is given by
\begin{equation*}
\psi(f)[t] = f(u_t \otimes \chi), \hspace{.3 cm} t \in [0,1],
\end{equation*}
where $f$ is applied using the continuous functional calculus. As before, $\chi \in \cC(\mathbb{S}^1)$ is the generator $\chi(z) = z$, $u_1 = I \otimes \chi$, $u_0 = I \otimes 1$, and $u_t$ is a continuous path of unitaries. The boundary conditions of the equivariant join are satisfied because
\begin{equation*}
\psi(f)[0] = f( (I \otimes 1) \otimes \chi) = (I \otimes 1) \otimes f \in \mathbb{C} \otimes \cC(\mathbb{S}^1)
\end{equation*}
and
\begin{equation*}
\psi(f)[1] = f( (I \otimes \chi) \otimes \chi) = f(\delta(I \otimes \chi)) \in C^*(\delta(A)) = \delta(A).
\end{equation*}
We have used the fact that $\delta$ is a unital $*$-homomorphism. If $\delta_\Delta$ denotes the coaction on $A \circledast_\delta  \cC(\mathbb{S}^1)$ given by applying $\mathrm{id} \otimes \Delta: A \otimes \cC(\mathbb{S}^1) \to A \otimes \cC(\mathbb{S}^1) \otimes \cC(\mathbb{S}^1)$ at each $t \in [0,1]$, then $\psi$ is $(\Delta, \delta_\Delta)$ equivariant. Indeed, the equations $\Delta(\chi) = \chi \otimes \chi$ and
\begin{equation*}
(\mathrm{id} \otimes \Delta)(\psi(\chi)[t]) = (\mathrm{id} \otimes \Delta)(u_t \otimes \chi) = u_t \otimes \chi \otimes \chi = \psi(\chi)[t] \otimes \chi
\end{equation*}
show that the unital $*$-homomorphisms $\delta_\Delta \circ \psi$ and $(\psi \otimes \mathrm{id}) \circ \Delta$ agree on $\chi$, the generator of the domain $\cC(\mathbb{S}^1)$. 

\Cref{th.type2failure} leads to a wider class of counterexamples to \Cref{conj.BDH2} Type 2. First, we make the following easy observation.

\begin{lemma}\label{le.prod}
  If \Cref{conj.BDH2} Type 2 fails for the compact groups $G_i$, $i\in I$, then it fails for the product $G=\prod_I G_i$. 
\end{lemma}
\begin{proof}
  According to \Cref{le.tech}, the hypothesis ensures the existence of free actions of $G_i$ on $C^*$-algebras $A_i$ together with equivariant morphisms $\varphi_i:\cC(G_i)\to A_i$ that are contractible to one-dimensional representations of $\cC(G_i)$, respectively.  

Now we can apply \Cref{le.tech} again in the opposite direction, using the $C^*$-algebra $A=\bigotimes_I A_i$ (universal tensor product) with the obvious tensor product coaction by
\begin{equation*}
  \cC(G) \cong \bigotimes_I \cC(G_i)
\end{equation*}
and the equivariant map
\begin{equation*}
  \varphi=\bigotimes_I \varphi_i:\cC(G)\to A. 
\end{equation*}
The contractibility of $\varphi$ follows from that of the individual $\varphi_i$. 
\end{proof}

Failure of the conjecture for tori is then immediate.

\begin{corollary}\label{cor.type2failure}
  \Cref{conj.BDH2} Type 2 fails for tori $\bT^I$.
\end{corollary}
\begin{proof}
  This follows from applying \Cref{th.type2failure} and \Cref{le.prod} to $\bS^1$.
\end{proof}

The counterexample presented in \Cref{th.type2failure} uses a $C^*$-algebra which is highly non-nuclear, but this can be avoided.

\begin{theorem}
  There is a counterexample to \Cref{conj.BDH2} Type 2 for which
  $H = \cC(\bS^1)$ and $A$ is a unital, separable, nuclear
  $C^*$-algebra.
\end{theorem}
\begin{proof}
  Let $B$ be a unital, separable, nuclear $C^*$-algebra with
  $K_0(B) \cong \{0\} \cong K_1(B)$, such as the Cuntz Algebra $\mathcal{O}_2$ \cite[Theorems 3.7 and 3.8]{cuntzalgebras}. Since the $K$-groups of $B$ are trivially torsion free and the commutative $C^*$-algebra $\cC(\bS^1)$ is certainly in the bootstrap class, the K{\"u}nneth formula \cite[Theorem 23.1.3]{blackadar} shows that
  $K_1(B \otimes \cC(\bS^1)) \cong (K_0(B) \otimes K_1(\cC(\bS^1))) \oplus (K_1(B) \otimes K_0(\cC(\bS^1))) \cong \{0\}$. Let $\chi \in \cC(\bS^1)$
  denote the standard generating character and fix $n$
  such that there is a path of unitaries in
  $M_n(B \otimes \cC(\bS^1))$ connecting $(1 \otimes \chi) \oplus I_{n-1}$ to $I_n$. Since the matrix unitary group $U_n(\bC)$ is path
  connected, it follows that $I_1 \oplus (1 \otimes \chi) \oplus I_{n-2}$,
  $I_2 \oplus (1 \otimes \chi) \oplus I_{n-3}$, \ldots, $I_{n-1} \oplus (1 \otimes \chi)$ are also connected to the identity, as is their
  product $\bigoplus\limits_{i=1}^n 1 \otimes \chi$. That is, in the separable, unital, nuclear $C^*$-algebra
  $A := B \otimes \cC(\bS^1) \otimes M_n(\bC)$,
  $v_1 := 1 \otimes \chi \otimes I_n$ is connected via a path of
  unitaries $v_t$ to the identity element $v_0$.

Let $\bS^1$ act on $A$ via rotation in the $\cC(\bS^1)$ tensorand, so that $v_1$ is a $\chi$-eigenvector for this action. Then
\begin{equation*}
\phi_t: \chi \in \cC(\bS^1) \mapsto v_t \in A
\end{equation*}
defines a continuous path of morphisms. Since $\chi$ and $v_1$ are
$\chi$-eigenvectors in $\cC(\bS^1)$ and $A$, respectively, $\phi_1$ is
equivariant. Further, $\phi_0$ is a $1$-dimensional representation
since $v_0$ is the identity element. Therefore, the conditions of
\Cref{le.tech} are satisfied, and the counterexample follows.
\end{proof}

Despite the above counterexamples, there are some circumstances under which Type 2 of \Cref{conj.BDH2} holds. See for example \cite[Corollary 2.7]{hajac}, in which the authors show that Type 2 holds for the compact quantum group $\cC(SU_q(2))$ acting on its iterated joins, based on a general result about finite-dimensional representations. In a different vein, the next theorem is dual to the following topological argument: an equivariant map $X * G \to G$ is automatically surjective, and $X * G$ is connected, so such a map cannot exist if $G$ is disconnected. To adapt this picture to the fully noncommutative setting (when neither $A$ nor $H$ is abelian), note that a compact group $G$ is disconnected precisely when it admits a nontrivial finite quotient $G\to L$. In turn, this corresponds to an embedding $\cC(L)\to \cC(G)$ of a finite-dimensional Hopf $C^*$-algebra into that of $G$. For this reason, we regard such an embedding in the quantum case as an analogue of disconnectedness.

\begin{theorem}\label{th.disc}
  Suppose a compact quantum group $H$ admits an equivariant embedding $K \hookrightarrow H$ of a nontrivial compact quantum group $(K,\Delta)$ whose underlying Hopf $C^*$-algebra is finite-dimensional. Then \Cref{conj.BDH2} Type 2 holds for any free coaction $\delta:A\to A\otimes H$.
\end{theorem}
\begin{proof}
  Suppose we have an $H$-equivariant map $\psi:H\to A\circledast_\delta  H$. The image of $K$ must be contained in
  \begin{equation}\label{eq:isotyp}
    \bigoplus_\rho (A\circledast_\delta  H)_\rho \subset \bigoplus_\rho (\cC([0,1],A\otimes H))_\rho
  \end{equation}
as $\rho$ ranges over the irreducible $K$-comodules. Since the $H$-coaction in \Cref{eq:isotyp} is the regular one on the $H$-tensorand, the right hand side of \Cref{eq:isotyp} is simply $\cC([0,1],A\otimes K)$.

Since $K$ is a finite-dimensional Hopf $C^*$-algebra, there exists a counit $\varepsilon:K\to \bC$. Applying $\varepsilon$ to the $K$ tensorand of 
\begin{equation*}
  \psi(K)\subset \cC([0,1],A\otimes K) 
\end{equation*}
yields a path $\psi_t:K\to A$ of $C^*$-algebra morphisms such that $\psi_1$ is $H$-equivariant while $\psi_0$ takes values in $\bC\subset A$. 

Now, $K$ is a finite-dimensional $C^*$-algebra, and hence it has a $\bC$-basis consisting of finite-order unitaries. Moreover, $H$-equivariance ensures that for some finite-order unitary $u\in K$, the element $\psi_1(u)\in A$ is not a scalar. But then the spectrum of $\psi_1(u)$ is a non-trivial finite subgroup of $\bS^1$, and the continuity of the spectrum for the norm topology on the unitary subgroup of $A$ implies that $\psi_1(u)$ cannot be connected by a path to $\psi_0(u)\in \bC\subset A$. We have reached a contradiction, so there can be no equivariant map $\psi$.  
\end{proof}

Restricting our attention once again to the commutative case $H = \cC(G)$, we know from \Cref{th.disc} that  there are no counterexamples to Type 2 for which $G$ has a nontrivial finite quotient group, so $G$ must be connected to produce a counterexample. Further, there are counterexamples for all tori $G = \bT^I$, which of course are path connected, from \Cref{cor.type2failure}.

\begin{question}\label{ques}
If $G$ is a compact, abelian, (path) connected group, must a Type 2 counterexample exist for $H = \cC(G)$?
\end{question}

A compact abelian group $G$ is connected if and only if its discrete abelian Pontryagin dual $\Gamma = \widehat{G}$ is torsion-free (\cite[Corollary 7.70]{compactgroups}). Pontryagin duality is expressed in the identity $\cC(G) \cong C^*(\Gamma)$, and for any discrete group $\Gamma$, abelian or not, there is a natural way to make $C^*(\Gamma)$ into a compact quantum group using the comultiplication
\begin{equation*}
\Delta: \sum a_\gamma \gamma \in C^*(\Gamma) \mapsto \sum a_\gamma \hspace{2pt} \gamma \otimes \gamma \in C^*(\Gamma) \otimes C^*(\Gamma).
\end{equation*}

We will see below that the analogue of \Cref{ques} in the discrete
torsion-free (non-abelian) setting can be resolved in the negative. That is, there are compact quantum groups of the form $C^*(\Gamma)$ with
torsion-free non-abelian $\Gamma$ for which \Cref{conj.BDH2} Type 2
holds. Moreover, the groups $\Gamma$ in question will be amenable, so
that there will be no ambiguity regarding which $C^*$ completion
$C^*(\Gamma)$ we are considering.

Let $n\ge 2$ be a positive integer, and equip the torus
$\bT^n\cong \bR^n/\bZ^n$ with an automorphism $\sigma$ regarded
simultaneously as an element of $SL(n,\bZ)$ fixing the lattice
$\bZ^n\cong \pi_1(\bT^n)$. Throughout the rest of \Cref{se.tp2} we
assume that $\sigma$ is {\it hyperbolic}, i.e. its eigenvalues as an
element of $SL(n,\bZ)$ have absolute value not equal to $1$. For
background on hyperbolic automorphisms on smooth manifolds we refer to
\cite[Chapter 6]{ergd-hyp}; in reference to the eigenvalue condition
on $\sigma$, see in particular \cite[Definition 6.3 and Exercise
6.2]{ergd-hyp}. Hyperbolicity also implies that $\sigma$ is {\it
  expansive} in the sense that there is an $\varepsilon>0$ such that
\begin{equation}\label{eq:exp}
  \bT^n\ni x\ne y\Rightarrow \mathrm{sup}_{n\in \bZ}~d(\sigma^n(x),\sigma^n(y))>\varepsilon
\end{equation}
for any distance function $d$ on the torus (see e.g. \cite[Definition
5.5]{ergd-hyp}).

Now, $\sigma$ induces an automorphism $\hat\sigma$ on the Pontryagin-dual
lattice $\bZ^n\cong \widehat{\bT^n}$, allowing us to construct the
extension $\bZ^n\rtimes_{\hat\sigma}\bZ$.

 \begin{lemma}\label{le.hyp}
   If $\sigma$ as above is hyperbolic, then $\bZ^n\rtimes_{\hat\sigma}\bZ$ is
   amenable and torsion-free.
\end{lemma}
\begin{proof}
  First, $\bZ^n\rtimes_{\hat\sigma}\bZ$ is certainly amenable, as it
  is an extension of two abelian groups. Writing a generic nonzero
  element $x$ of $\bZ^n\rtimes_{\hat\sigma}\bZ$ as $y\sigma^l$ for
  some
  \begin{equation*}
    (0,0)\ne (y,l)\in \bZ^n\times\bZ,
  \end{equation*}
  the $k^{th}$ power of $x$ is
  \begin{equation*}
    (y+\sigma^l y+\cdots+\sigma^{l(k-1)}y)\sigma^{lk}.
\end{equation*}
This element is nontrivial, as for nonzero $l$,
$\sigma^l\in SL(n,\bZ)$ has no root-of-unity eigenvalues.
\end{proof}

Since $\bZ^n\rtimes_{\hat\sigma}\bZ$ as above is torsion-free, the
compact quantum group $C^*(\bZ^n\rtimes_{\hat\sigma}\bZ)$ does not fit
within the framework of \Cref{th.disc}. However, \Cref{conj.BDH2} Type
2 still holds for $C^*(\bZ^n\rtimes_{\hat\sigma}\bZ)$.

\begin{theorem}\label{th.hyp}
 Let $\bZ^n\rtimes_{\hat\sigma}\bZ$ be as in \Cref{le.hyp}. Then \Cref{conj.BDH2} Type 2 holds for
  the compact quantum group $C^*(\bZ^n\rtimes_{\hat\sigma}\bZ)$.
\end{theorem}
\begin{proof}
  Let $\Gamma := \bZ^n\rtimes_{\hat\sigma}\bZ$, and suppose a
  counterexample {\it does} exist for $C^*(\Gamma)$, in the form of a
  free coaction $\delta: A\to A\otimes C^*(\Gamma)$ and an equivariant
  morphism
  \begin{equation*}
    \psi:C^*(\Gamma)\to A\circledast_\delta C^*(\Gamma).
  \end{equation*}
  Regard $\psi$ as a path $\psi_t$, $t\in [0,1]$, as in \Cref{eq:nc},
  and consider the $C^*$-subalgebra $C^*(\bZ^n) \cong \cC(\bT^n)$ of
  $C^*(\Gamma)$. Restricting $\psi_t$ produces a path of equivariant
  morphisms $\phi_t: \cC(\bT^n) \to A \otimes \cC(\bT^n)$, where
  $\cC(\bT^n)$ acts on the right tensorand of the codomain. The
  boundary conditions
  \begin{equation*}
    \textrm{Ran}(\phi_0) \subseteq \bC\otimes \cC(\bT^n) \text{ and } \textrm{Ran}(\phi_1) \subseteq (A\otimes \cC(\bT^n)) \cap \delta(A) \subset B \otimes \cC(\bT^n)
  \end{equation*}
  are also satisfied, where $B$ is the closed direct sum of the
  isotypic subspaces of $A$ corresponding to $\bZ^n \subset
  \Gamma$. Applying the counit to the right tensorand produces a path
  of morphisms $\cC(\bT^n) \to A$ connecting a character
  $\cC(\bT^n) \to \bC$ to an equivariant morphism $\cC(\bT^n) \to B$,
  which must be injective. To show injectivity, we may first factor
  through the commutative range $C^*$-algebra to obtain
  $\cC(\bT^n) \to \cC(X) \hookrightarrow B$. Then we note that because
  the original morphism is equivariant, there is a corresponding
  coaction of $\cC(\bT^n)$ on $\cC(X)$, i.e. an action of $\bT^n$ on
  $X$. Finally, any equivariant continuous map $X \to \bT^n$ is
  surjective by direct examination of an orbit, so any equivariant
  morphism $\cC(\bT^n) \to \cC(X)$ is injective.

  The joint spectrum of the images of the $n$ unitary generators of
  $C^*(\bZ^n) \cong \cC(\bT^n)$ changes along a path $\mathrm{sp}_t$,
  $t\in [0,1]$ (continuous in the Hausdorff topology on closed subsets
  of $\bT^n$) from a singleton $\mathrm{sp}_0 = \{p\}$ to
  $\mathrm{sp}_1=\bT^n$. Moreover, because all of the homomorphisms
  are restrictions from $C^*(\Gamma)$, all of the spectra
  $\mathrm{sp}_t$ are $\sigma$-invariant. The eigenvalues of $\sigma$
  as an element of $SL(n,\bZ)$ have absolute values not equal to $1$,
  ensuring via the expansivity condition \Cref{eq:exp} the existence
  of some $r>0$ such that the orbit of any non-trivial element of the
  open ball $B_r(p)$ (in the standard Euclidean metric on
  $\bT^n \cong \bR^n / \bZ^n$) under $\bZ$ intersects the complement
  of $B_r(p)$. On the other hand, for small $t$ the spectrum
  $\mathrm{sp}_t$ will be contained in the open ball $B_r(p)$, which
  is a contradiction.
\end{proof}


\section{Type 1: Reductions and Possible Approaches}\label{se.tp1}

Corollary 2.4 of \cite{saturated} implies that Type 1 of
\Cref{conj.BDH2} holds when $H = \cC(G)$ for $G$ a compact
group with a nontrivial torsion element. Here we show a reduction of
the remaining classical case $H = \cC(G)$, with some comments on how
it might be proved. The problem reduces easily to subgroups of $G$, so
we may certainly assume our compact torsion-free $G$ is abelian, and
it is well-known that a copy of $\bZ_p = \varprojlim \bZ/p^n$, for
some prime $p$, embeds into $G$. We include a proof in order to glean
some information about the Pontryagin dual.

\begin{proposition}\label{pr.Zp_sbgp}
  Suppose $G$ is a nontrivial compact, torsion-free, abelian group,
  and fix any nontrivial character $1 \not= \tau \in
  \widehat{G}$. Then there is an embedding of $\bZ_p$ into $G$ for
  which the restriction of $\tau$ to $\bZ_p$ is nontrivial.
\end{proposition}
\begin{proof}
  First, we may replace $G$ with the compact subgroup $H$ generated by
  a single element $g$ with $\tau(g)$ nontrivial. Every character on
  $H$ is uniquely determined by its value at $g$, and hence the
  Pontryagin dual $\widehat{H} = \Gamma$ can be identified with a
  subgroup of the discrete circle group $\bS^1$. Since $H$ is
  torsion-free, $\Gamma$ has no nontrivial finite quotients. The Pontryagin dual of $\bZ_p$ is the group
  $\bZ/p^\infty$ of roots of unity whose orders are powers of $p$, so it
  suffices to prove that there is a surjection
  $\Gamma \to \bZ/p^\infty$ which does not annihilate $\tau|_H$.

  Regarded as a discrete abelian group, $\bS^1$ is the direct sum of
  one copy of $\bZ/p^\infty$ for each prime $p$, as well as continuum
  many copies of $\bQ$: 

\begin{equation}\label{eq:circ}
  \bS^1\cong \bigoplus_p (\bZ/p^\infty)\oplus \bQ^{\oplus 2^{\aleph_0}}.
\end{equation}
Since $\Gamma$ embeds into $\bS^1$, $\tau|_H$ has nontrivial image
under a map from $\Gamma$ to one of the summands in
\Cref{eq:circ}. There are now two cases to consider.

{\bf (1) A morphism $\Gamma\to \bZ/p^\infty$ does not annihilate
  $\tau|_{H}$.} Since $\Gamma$ has no nontrivial
finite quotients, and there are no proper infinite subgroups of
$\bZ/p^\infty$, we obtain the desired surjection
$\Gamma\to \bZ/p^\infty$.

{\bf (2) A morphism $\Gamma\to \bQ$ does not annihilate $\tau|_H$.} If
necessary, we may rescale $\bQ$ so that $\tau|_H$ is not mapped into
$\bZ$. We then have a map
\begin{equation*}
  \Gamma\to \bQ/\bZ\cong \bigoplus_p (\bZ/p^\infty)
\end{equation*}
which does not annihilate $\tau|_H$. By selecting an appropriate
summand, we can now continue as in case (1), completing the proof.
\end{proof}

An embedding $\mathbb{Z}_p \hookrightarrow G$ provides a reduction of \Cref{conj.BDH} and the classical subcase of \Cref{conj.BDH2} Type 1.

\begin{lemma}\label{le.red}
  If \Cref{conj.BDH} holds for the compact groups $\bZ_p$ for all
  primes $p$, then it holds in general. Similarly, if
  \Cref{conj.BDH2}, Type 1 holds for all $H = \cC(\bZ_p)$, then it
  holds whenever $H = \cC(G)$ for some nontrivial compact group $G$.
\end{lemma}
\begin{proof}
  Let $X$, $G$, $A$, and $H = \cC(G)$ be as in \Cref{conj.BDH} and
  \Cref{conj.BDH2}. If $G$ has a nontrivial torsion element, then the
  conjectures already hold by \cite[Corollary 3.1]{volovikovindex} and \cite[Corollary 2.4]{saturated}. Otherwise, a $G$-equivariant map $X*G\to X$
  restricts to a $\bZ_p$-equivariant map $X * \bZ_p \to X$ for some
  subgroup $\bZ_p\le G$ provided by \Cref{pr.Zp_sbgp}. Similarly, a
  $G$-equivariant morphism $A \to A \circledast \cC(G)$ restricts to a
  $G$-equivariant morphism $A \to A \circledast \cC(\bZ_p)$,
  where we note that the classical join and equivariant join are equivariantly isomorphic because of \Cref{eq.2joins}.
\end{proof}

As in \cite[Lemma 2.5]{saturated}, one way to approach the Type 1 conjecture is by iterating any proposed equivariant map $A \to A \circledast \cC(\bZ_p)$ with its joins $A \circledast \cC(\bZ_p) \to A \circledast \cC(\bZ_p) \circledast \cC(\bZ_p)$, etc., producing a chain
\begin{equation}\label{eq.somuchiteration}
A \to A \circledast \cC(\bZ_p) \to A \circledast \cC(\bZ_p) \circledast \cC(\bZ_p) \to A \circledast \cC(\bZ_p) \circledast \cC(\bZ_p) \circledast \cC(\bZ_p) \to \ldots
\end{equation}
of equivariant maps. Compositions then give
\begin{equation*}
A \to A \circledast \cC(\bZ_p) \circledast \cC(\bZ_p) \circledast \cdots \circledast \cC(\bZ_p) \cong A \circledast \cC(\bZ_p^{*n}) \hspace{.2 in} \forall n
\end{equation*}
with equivariant quotients
\begin{equation}\label{eq.iterationtrick}
A \to \cC(\bZ_p^{*n}) \hspace{.2 in} \forall n.
\end{equation}

Freeness of a $\bZ_p$-action on $A$ implies that the saturation
condition of \Cref{prop.saturation} (or, as $\bZ_p$ is abelian, the condition described in \cite[$\S$5]{phillips}) is met; that is, for any character
$\tau \in \widehat{\bZ_p}$ and eigenspace $A_\tau$,
$1 \in \overline{A_\tau A_\tau^*}$. Therefore, there is a finite $m$
and a list $f_1, \ldots, f_m, g_1, \ldots, g_m \in A_\tau$ such that
$\sum_{i=1}^m f_i g_i^*$ is invertible in $A$. The images of these
functions under \Cref{eq.iterationtrick} then show that if an equivariant map $A \to A \circledast \cC(\bZ_p)$ is assumed, then for this fixed
$m$, and for all $n$, there is a list of $m$ $\tau$-eigenfunctions in
$\cC(\bZ_p^{*n})$ with no common zeroes. It would suffice to show
that for some $n$ much larger than $m$, this cannot happen. However,
unlike in the torsion case, it is not obvious if such a contradiction
actually occurs. It is therefore prudent to study the constraints of
equivariant maps $\bZ_p^{*n} \to \bC^{m} \setminus \{0\}$, where
$\bZ_p$ acts on $\bC^{m}$ through a particular character, as $n$
increases. This may be recast as a question about actions of finite
groups.

\begin{lemma}\label{le.approx}
Suppose $G=\varprojlim G_\alpha$ is a filtered limit of compact groups, with a matching limit $X=\varprojlim X_\alpha$ of compact Hausdorff spaces indexed by the same filtered poset. The ordering is such that $\alpha\ge \beta$ implies that there exist maps $X_\alpha\to X_\beta$ and $G_\alpha \to G_\beta$. Further, assume that each $G_\alpha$ acts continuously on $X_\alpha$ so that the diagram
\begin{equation}\label{eq:comp_act}
  \begin{tikzpicture}[auto,baseline=(current  bounding  box.center)]
    \path[anchor=base] (0,0) node (aa) {$G_\alpha\times X_\alpha$} +(2,.5) node (a) {$X_\alpha$} +(4,0) node (b) {$X_\beta$} +(2,-.5) node (bb) {$G_\beta\times X_\beta$};
         \draw[->] (aa) to[bend left=10]  (a);
         \draw[->] (aa) to[bend right=10]  (bb);
         \draw[->] (a) to[bend left=10]  (b);
         \draw[->] (bb) to[bend right=10]  (b);
  \end{tikzpicture}
\end{equation}
commutes for all $\alpha \geq \beta$. Let $E$ be a finite-dimensional unitary representation of a fixed $G_\beta$. Further, let $G$ and $G_\alpha$, $\alpha \geq \beta$, also act on $E$ through the quotient maps $\pi_\beta: G \to G_\beta$ and $\pi_{\alpha \beta}: G_\alpha \to G_\beta$, respectively. Then for any continuous $G$-equivariant map $f: X \to E$ and $\varepsilon > 0$, there exist an $\alpha \geq \beta$ and a continuous $G_\alpha$-equivariant map $f_\alpha: X_\alpha \to E$ such that $\|f- f_\alpha \circ \pi_\alpha\|_\infty<\varepsilon$.
\end{lemma}
\begin{proof}
  Approximation by {\it arbitrary} functions $X_\alpha\to E$ follows from the Stone-Weierstrass Theorem: the complex-valued, continuous functions on $X$ that factor through some $X_\alpha$ (where the choice of $\alpha$ may depend on the function) form a complex unital $*$-algebra that separates points, because $X=\varprojlim X_\alpha$. It follows that for any $\varepsilon>0$, we can find some $\psi_\alpha:X_\alpha\to E$ for which 
\begin{equation}
  \label{eq:epsilon}
\|f-\psi_\alpha \circ \pi_\alpha\|_\infty<\varepsilon,
\end{equation}
and we may also assume $\alpha \geq \beta$. If $\mu$ denotes the Haar measure on $G$, then the averaging procedure 
\begin{equation*}
  \varphi\mapsto \varphi\triangleleft\mu :=\int_G g \varphi(g^{-1}\bullet)\ \mathrm{d}\mu(g) \in \cC(X,E)
\end{equation*}
produces $G$-equivariant maps, does not increase supremum norms, and
fixes $f$ (because $f$ is already $G$-equivariant). In conclusion,
applying $\triangleleft\mu$ to \Cref{eq:epsilon} produces
\begin{equation*}
  \|f-(\psi_\alpha \circ \pi_\alpha) \triangleleft\mu\|_\infty<\varepsilon.
\end{equation*}
Because $\alpha \geq \beta$ and the $G$-action on $E$ factors through
$\pi_\beta$, $(\psi_\alpha \circ \pi_\alpha)\triangleleft\mu$ must
factor through $\pi_\alpha$ as a composition
$f_\alpha \circ \pi_\alpha$. It follows that $f_\alpha$ is
$G_\alpha$-equivariant and
$\|f- f_\alpha \circ \pi_\alpha\|_\infty<\varepsilon$.
\end{proof}

The assumptions of \Cref{le.approx} are frequently satisfied when $G$
is a pro-$p$ group, such as $\bZ_p$, as any finite-dimensional unitary
representation $E$ of $G$ will factor through one of the $p$-group
quotients $G_\alpha$, such as $\bZ/p^k$. The role of $X_\alpha$ may
then be played by an iterated join of $G_\alpha$, whose inverse limit
is the iterated join of $G$. Moreover, following the techniques of
Dold in \cite{dold}, if we fix a primitive $p^k$ root of unity
$\omega$ and note that the iterated join $(\bZ/p^k)^{*2n}$ is
$(2n-2)$-connected, it follows that there is a map
$\bS^{2n-1} \to (\bZ/p^k)^{*2n}$ which is equivariant for the rotation
action $\vec{z} \mapsto \omega \vec{z}$ on the sphere and the diagonal
action of $\bZ/p^k$ on its iterated join.

Having replaced the relevant domain of an equivariant map with a
sphere, we may also view the unit sphere
$\bS^{2m-1} \subset \bC^m \setminus \{0\}$ as the codomain through
scaling, leading us to the following question.

\begin{question}\label{ques.solenoidal}
  For positive integers $k$, $r$, $m$, and $n$ with $k > r$, let
  $\bZ/p^k$ act on the sphere $\bS^{2n-1}$ freely through
  multiplication by $e^{2 \pi i / p^k}$ and non-freely on the sphere
  $\bS^{2m-1}$ through multiplication by $e^{2 \pi i / p^r}$. Call an
  equivariant map $\bS^{2n-1} \to \bS^{2m-1}$ for these actions a
  $(p^k, p^r)$-map.

\vspace{.12 in}

\noindent Fix $m$. Do there exist $n$ and $r$ such that for all $k > r$, $(p^k, p^r)$-maps $\bS^{2n-1} \to \bS^{2m-1}$ do not exist?
\end{question}

If for all $m \geq 2$ and primes $p$, the answer to
\Cref{ques.solenoidal} is yes, then the Type 1 conjecture immediately
follows, so resolution of \Cref{ques.solenoidal} is one possible
approach to the conjecture. These types of questions have been
studied, but to our knowledge no tight dimension bounds have been
published that apply for large $k$. Precise bounds have been found for
$k = 2$, $r = 1$, and scale information is known as $k \to \infty$. We
follow the notation of \cite{meyer}:
\begin{equation*}
v_{p,k}(n) = \mathrm{min} \{m \in \bZ^+: \textrm{there exists a } (p^k, p)\textrm{-map } \bS^{2n-1} \to \bS^{2m-1} \}.
\end{equation*}
Theorem 4.8 of \cite{meyer} states that for any odd prime $p$, $v_{p,2}(1) = 1$ and
\begin{equation}\label{eq.level}
\begin{array}{cccccccccc} \left\lceil \cfrac{n-2}{p} \right\rceil + 1 & \leq & v_{p, 2}(n) & \leq & \left\lceil \cfrac{n-2}{p} \right\rceil + 2 & \textrm{ for } & n \not\equiv 2 \textrm{ mod } p \\ \\ & & v_{p,2}(n) & = & \left\lceil \cfrac{n-2}{p} \right\rceil + 2 & \textrm{ for } & n \equiv 2 \textrm{ mod } p. \end{array}
\end{equation}
Similar bounds for $p = 2$ may be found in \cite{stolz}; the case $p = 2$ is unique, as the usual $\bZ/2$ antipodal action exists on even spheres as well as odd. Moreover, Theorem 5.5 of \cite{meyer} is equivalent to the claim that
\begin{equation}
\lim_{n \to \infty} \cfrac{v_{p,k}(n)}{n} = \cfrac{1}{p^{k-1}},
\end{equation}
so for large spheres $\bS^{2n-1}$, there is a large gap between $n$ and the
smallest $m$ with $(p^k, p)$-maps $\bS^{2n-1} \to \bS^{2m-1}$, roughly
$n/m \approx p^{k-1}$. This does not answer
\Cref{ques.solenoidal} for $r = 1$, as this behavior might only manifest
in large dimension, and the number dual to $v_{p,k}$ for fixed
\textit{codomain},
\begin{equation}
w_{p,k}(m) = \mathrm{max} \{n \in \bZ^+: \textrm{there exists a } (p^k, p)\textrm{-map } \bS^{2n-1} \to \bS^{2m-1} \},
\end{equation}
might remain bounded in $k$. It is also known from
\cite[Theorem 1.2]{bartsch} that
\begin{equation}\label{eq.bartsch}
v_{p,k}(n) \geq \left\lceil \cfrac{n-1}{p^{k-1}} \right\rceil + 1,
\end{equation}
so boundedness of the set $\{w_{p,k}(m)\}_{k \in \bZ^+}$ for
individual $m$ could follow, for instance, if the constant term $1$ in
\Cref{eq.bartsch} can be replaced with an unbounded, sublinear function of $n$ alone.

Even if \Cref{ques.solenoidal} has a negative answer, it is possible that the Type 1 conjecture may still hold. In what follows, we consider a related approach that is less demanding than \Cref{ques.solenoidal}.

\begin{proposition}
\Cref{conj.BDH} and the classical subcase $H = \cC(G)$ of \Cref{conj.BDH2} Type 1 are equivalent.
\end{proposition}
\begin{proof}
We follow the idea of \cite{saturated}. \Cref{conj.BDH2} Type 1 implies \Cref{conj.BDH}, so assume that \Cref{conj.BDH2} Type 1 fails for $H = \cC(G)$. By \Cref{le.red} there is then a free action $\alpha$ of $\bZ_p$ on a unital $C^*$-algebra $A$ and an equivariant morphism $A \to A \circledast \cC(\bZ_p)$. This implies that the largest commutative quotient $A \twoheadrightarrow B$ is nontrivial, and the action $\alpha$ descends to an action $\beta$ on $B$. From examination of spectral subspaces, we see that $\beta$ is free. The composition $A \to A \circledast \cC(\bZ_p) \twoheadrightarrow B \circledast \cC(\bZ_p)$ descends to an equivariant morphism $B \to B \circledast \cC(\bZ_p)$, as the codomain is commutative. This is dual to an equivariant map $X * \bZ_p \to X$ for the corresponding free action of $\bZ_p$ on $X \not= \emptyset$, so \Cref{conj.BDH} also fails.
\end{proof}

\begin{question}\label{ques.infiniteiteration}
Suppose $X$ is a compact Hausdorff space with a free action of $\bZ_p$. Equip $\bZ_p^{*n}$ with the diagonal action and take the direct limit $\bZ_p^{*\infty}$, with the diagonal action as well. Is it possible for an equivariant map $\bZ_p^{*\infty} \to X$ to exist?
\end{question}

If the answer is always no, then the Type 1 conjecture follows, as the topological iteration dual to \Cref{eq.somuchiteration} takes an equivariant map $X * \bZ_p \to X$ and produces a chain
\begin{equation*}
\cdots \to X * \bZ_p * \bZ_p \to X * \bZ_p \to X,
\end{equation*}
which gives equivariant maps $\bZ_p^{*n} \to X$ for all $n$. Similar to \cite[Alternative Proof B]{saturated}, these maps are consistent with the inclusions of $\bZ_p^{*n}$ into $\bZ_p^{*n+1} = \bZ_p * \bZ_p^{*n}$, so they extend to an (equivariant) map $\bZ_p^{*\infty} \to X$, which would give a contradiction based on the assumed negative answer to \Cref{ques.infiniteiteration}.

Because $\bZ_p^{*\infty}$ is a non-compact Tychonoff space, its Stone-\v{C}ech compactification $\beta \bZ_p^{*\infty}$ exists. Now, $\beta \bZ_p^{*\infty}$ is supplied with a \textit{not necessarily continuous} diagonal action of $\bZ_p$, as seen by following the universal property of the Stone-\v{C}ech compactification under each individual homeomorphism $x \mapsto x \cdot g$. Any equivariant map $\bZ_p^{*\infty} \to X$ extends to the Stone-\v{C}ech compactification, dual to a morphism $\cC(X) \to \cC_b(\bZ_p^{*\infty})$. Perhaps this is ruled out by limitations on the spectral subspaces.

\begin{question}\label{ques.eigenfunctionsinfiniteiteration}
Consider $\cC(\beta \bZ_p^{*\infty}) = \cC_b(\bZ_p^{*\infty})$ with the (not necessarily continuous) diagonal action of the compact group $\bZ_p$. For some $\tau \in \widehat{\bZ_p}$, does it hold that $1 \not\in \overline{\cC_b(\bZ_p^{*\infty})_\tau^* \hspace{2 pt} \cC_b(\bZ_p^{*\infty})_\tau}$?
\end{question}

\section{Correction of the Literature}

In a previous version of this manuscript, we proposed a solution to Type
1 of \Cref{conj.BDH2} using Claim 2.6 of \cite{tver}, which we repeat
here.

\begin{claim}[\cite{tver} Claim 2.6, Erroneous]\label{cl}
  Let $q \ge 2$ be a prime power, $d\ge 1$ an odd integer, and
  $N=(q-1)(d+1)$. For a group $G$ of order $q$, let $E$ be a unitary
  $G$-representation of (complex) dimension $N/2$ with no trivial
  subrepresentations. Then every $G$-equivariant map $G^{*(N+1)}\to E$
  has a zero.
\end{claim}

We are grateful to Robert Edwards for pointing out to us that this
claim is incorrect. Namely, the proof given in \cite{tver} only
concerns the case $G = \bigoplus\limits_{i=1}^k \bZ/p$, and for
$G = \bZ/p^k$ the result fails with a counterexample that may be found
by modifying \cite[p. 153-154]{wil} (due to
Floyd).

For $m \geq 2$, define a $(p^k, p)$-map $f$ on $\bS^{2m-1}$ in polar form by
\begin{equation*}
f: (z_1, \ldots, z_m) = (r_1 u_1, \ldots, r_m u_m) \in \bS^{2m-1} \mapsto (r_1 u_1^{p^{k-1}}, \ldots, r_m u_m^{p^{k-1}}) \in \bS^{2m-1},
\end{equation*}
so that $f$ has degree $p^{m(k-1)}$. Since
this degree is a multiple of $p^k$, $f$ may be modified along its free
orbits to produce a $(p^k, p)$-map $h: \bS^{2m-1} \to \bS^{2m-1}$
which is homotopically trivial. Let $X = \bS^{2m-1} * \bZ/p^k$ and
equip $X$ with the free diagonal action of $\bZ/p^k$, so that the
homotopically trivial map $h$ produces a $(p^k,p)$-map from $X$ to
$\bS^{2m-1}$. Further, $X$ is $(2m-1)$-connected, so there is an
equivariant map $(\bZ/p^k)^{*2m+1} \to X$, which may be constructed
cell-by-cell as in \cite{dold}. Composition of these maps gives a
$(p^k,p)$-map $(\bZ/p^k)^{*2m+1} \to \bS^{2m-1}$.

Finally, if $m \geq 2$ and $E = \bC^m$ is equipped with the
representation of $\bZ/p^k$ through an order $p$ character, there is
an equivariant map $(\bZ/p^k)^{*2m+1} \to E \setminus \{0\}$. Since
$E$ has no trivial subrepresentations, this produces many
counterexamples to the claim, such as for $p = 2$, $k = 2$, $q = 4$,
$m = 3$, $d = 1$, $N = 6$. Moreover, \Cref{cl} is actually
inconsistent with results in \cite{meyer} and \cite{stolz}, and a
reasonably explicit construction of a $(4,2)$-equivariant
counterexample may be found by manipulating \cite[Example
5.1]{jaworowski}, which produces a $(4,2)$-map $\bS^3 \to \bS^2$ and
consequently a $(4,2)$-map $\bS^3 \to \bS^3$ which is homotopically
trivial.

The main issue with \Cref{cl} can be traced back to the proof in
\cite{tver} for elementary abelian $p$-groups, making it clear how
this proof fails to apply to other classes of $p$-groups (and
specifically to $G=\bZ/p^k$, the case we are interested in here). The
line of reasoning followed in \cite{tver} translates \Cref{cl} into
the language of the equivalent \cite[Claim 4.9]{tver}, which in our
case demands the following.

\begin{claim}[Erroneous]\label{cl-bis}
  Let $G=\bZ/p^k$, $\chi$ a character of $G$ with kernel
  $\bZ/p^{k-1}$, and $E=\chi^{\oplus\frac N2}$ for $N$ as in \Cref{cl}. Then the top Chern class of the bundle on $G^{*(N+1)}/G$ associated
  to $E$ is non-vanishing.
\end{claim}

However, examining the proof following \cite[Claim 4.9]{tver}, we see
that in fact the Chern class in question is $(p^{k-1}y)^{\frac N 2}$
in the cohomology ring
\begin{equation*}
  H^*(\bZ/p^k,\bZ) \cong \bZ[y]/(p^ky),\quad \mathrm{deg}(y)=2.
\end{equation*}
For $k\ge 2$ and $N\ge 2$ this vanishes because $p^k$ divides
$(p^{k-1})^{\frac N2}$. In conclusion, we have the following strong
negation of \Cref{cl,cl-bis}, preserving the notation and conventions
therein.

\begin{proposition}
  Let $q$ and $N$ be as in \Cref{cl}, where $q$ is not prime. Then for
  an order-$p$ character $\chi$ of $\bZ/q$, there exists a
  nowhere-vanishing $(\bZ/q)$-equivariant map
  $(\bZ/q)^{*(N+1)}\to \chi^{\oplus\frac N2}$.
\end{proposition}


\section*{Acknowledgments}
We are deeply grateful to Robert Edwards for his pivotal correction. B.P. would also like
to thank his Ph.D. advisors John McCarthy and Xiang Tang, his postdoctoral mentors Baruch Solel and Orr Shalit, and Piotr M. Hajac for his hospitality at IM PAN.


\def\polhk#1{\setbox0=\hbox{#1}{\ooalign{\hidewidth
  \lower1.5ex\hbox{`}\hidewidth\crcr\unhbox0}}}
  \def\polhk#1{\setbox0=\hbox{#1}{\ooalign{\hidewidth
  \lower1.5ex\hbox{`}\hidewidth\crcr\unhbox0}}}

\end{document}